\documentclass[letterpaper,10pt,reqno,final]{amsart}
\usepackage{amsmath}%
\usepackage{amsfonts}%
\usepackage{amssymb}%
\usepackage{graphicx}
\usepackage{functan}
\usepackage{mathrsfs}
\usepackage[all,cmtip]{xy}
\usepackage[obeyspaces]{url}
\usepackage{braket}
\usepackage{tikz}
\usetikzlibrary{arrows}
\tikzstyle{block}=[draw opacity=0.7,line width=1.4cm]
\newtheorem{theorem}{Theorem}

\newtheorem{example}{Example}

\newtheorem{lemma}{Lemma}

\newtheorem{remark}{Remark}

\numberwithin{equation}{section}
\numberwithin{theorem}{section}
\numberwithin{lemma}{section}
\numberwithin{corollary}{section}
\numberwithin{definition}{section}
\numberwithin{example}{section}
\numberwithin{remark}{section}
\begin{document}
\title[Matrix Schr{\"o}dinger Unitary Groups]{On Matrix Schr{\"o}dinger Unitary Groups in Particular Representations of Finite Dimensional Quantum Dynamical Systems}
\author{Fredy Vides}
\address
{Escuela de Matem{\'a}tica y Ciencias de la Computaci{\'o}n \newline%
\indent Universidad Nacional Aut{\'o}noma de Honduras}%
\email{fvides@unah.edu.hn}
\urladdr{http://fredyvides.6te.net}
\thanks{This research has been performed in part thanks to the financial support of the School of Mathematics and Computer Science of the National Autonomous University of Honduras.}
\date{\today}

\begin{abstract}
In this paper we study some particular types of matrix Schr{\"o}dinger unitary groups of the form $\exp(-it\mathbb{H})$ where $\mathbb{H}\in M_N(\mathbf{C})$ is the Hamiltonian of a given quantum dynamical system modeled in the finite dimensional Hilbert space $\mathcal{H}$. Once we have defined a particular matrix Schr{\"o}dinger unitary group we perform some estimates for its approximation and its corresponding implementation in the numerical solution of the finite dimensional Schr{\"o}dinger evolution equation to that it is related.
\end{abstract}
\maketitle

\section{Introduction} \label{intro}

In this work we will focus our attention in the study some properties of Matrix Schr{\"o}dinger semigroups that will be described in general by the set $\{\mathbb{S}_t:=\exp(-it\mathbb{H}):t\in \mathbf{R}\}$, where $\mathbb{H}\in\mathscr{L}(\mathcal{H})$ is the hamiltonian of a prescribed quantum dynamical system and $\mathcal{H}:=\mathcal{H}(G)$ is a finite dimensional Hilbert space related in some suitable sense to a weighted graph $G:=(V_G,E_G)$ that will be called representative graph, the elements of the semigroup clearly satisfy the conditions: (i) $S_0=\mathbf{1}_{\mathcal{H}}$, (ii) $\mathbb{S}_t \cdot \mathbb{S}_s(\cdot)=\mathbb{S}_{t+s}(\cdot)$ and (iii) $\lim_{h\to0^+}\mathbb{S}_h \phi=\phi, \forall \mathbf{\phi}\in \mathcal{D}\subseteq \mathcal{H}$, besides the condtion (iv) $\norm{}{\mathbb{S}_t \psi_0}=\norm{}{\psi_0}$, will be also satisfied when $\mathbb{H}$ is self adjoint. The Hamiltonian $\mathbb{H}\in\mathscr{L}(\mathcal{H})$ presented above is related to a prescribed quantum dynamical system trough the evolution equation given by:
\begin{equation}
E \ket{\psi(t)}=\mathbb{H}\ket{\psi(t)}
\label{abs_ev}
\end{equation}
with $\psi(0)=\psi_0\in\mathcal{H}$ and where $E\longrightarrow i/\hbar D_t$, here $\mathbb{H}\in\mathscr{L}(\mathcal{H})$ will in general have the form $\mathbb{H}=\mathbb{H}_0+\mathbb{H}_1$ with $\mathbb{H}_0\in M_{N}(\mathbf{C})$ self-adjoint, and with $\mathbb{H}_1\in M_{N}(\mathbf{C})$ diagonal, for simplicity, in this work we will consider our scale such that $\hbar=1$.

In the following sections we will implement some operator theory techniques in the theoretical analysis of the approximation schemes for the matrix Schr{\"o}dinger semigroups and in the end some numerical implementations will be presented.

\section{Particular Representations of Finite Dimensional Quantum Dynamical Systems}

In this section we will present the main ideas behind particular representation techniques of finite dymensional quantum dynamical systems.

\subsection{Generalized Matrix Form of Kets and Bras} In this work we will consider that the space of states $\mathcal{H}$ of a finite dimensional quantum dynamical system, is a finite dimensional Hilber space, whose inner product is induced by an inner product matrix $\mathbb{M}_\mathbb{H}\in M_N(\mathbf{C})$, wich is a symmetric positive definite matrix that satisfies the relation
\begin{equation}
\braket{u|v}:=u^\ast \mathbb{M}_\mathbb{H} v
\end{equation}
since $\mathbb{M}_\mathbb{H}$ is symmetric positive definite, we can obtain a factorization of the form $\mathbb{M}_\mathbb{H}:=\mathbb{W}_\mathbb{H}^\ast \mathbb{W}_\mathbb{H}$, where $\mathbb{W}_\mathbb{H}\in M_N(\mathbf{C})$ represents the formal square root of $\mathbb{M}_\mathbb{H}$. From the above relations we can obtain matrix representations for kets, bras and induced norm by $\braket{\cdot|\cdot}$ operations according to the rules:
\begin{eqnarray}
\ket{u}&\longleftrightarrow& u\\
\bra{u} &\longleftrightarrow& u^\ast \mathbb{M}_\mathbb{H}\\
\norm{}{u}&\longleftrightarrow&  \braket{u|u}^{1/2}.
\end{eqnarray}
from the definition of the norm operation, and if we denote by $\scalprod*{2}{\cdot}{\cdot}$ the usual inner product in $\mathbb{C}^N$ given by $\scalprod*{2}{x}{y}:=y^\ast x$ it can be seen that $\braket{x|y}=\scalprod*{2}{y}{x}$ and also that
\begin{equation} 
\norm{}{u}:=\scalprod*{2}{\mathbb{W}_\mathbb{H}u}{\mathbb{W}_\mathbb{H}u}^{1/2}=\norm*{2}{\mathbb{W}_\mathbb{H}u},
\end{equation}
in a similar way the induced matrix norm in $M_N(\mathbf{C})$ by $\norm{}{\cdot}$ can be expressed in the form
\begin{equation}
\norm{}{\mathbb{A}}:=\sup_{\norm{}{u}=1}\norm{}{\mathbb{A}u}=\norm*{2}{\mathbb{W}_\mathbb{H}\mathbb{A}\mathbb{W}_\mathbb{H}^{-1}};
\end{equation}
from the relation between $\braket{\cdot|\cdot}$ and $\scalprod*{2}{\cdot}{\cdot}$, it can be seen that for a given matrix $\mathbb{A}\in M_N(\mathbf{C})$, one can compute its adjoint $\mathbb{A}^\dagger$ with respect to $\braket{\cdot|\cdot}$ using the following expression
\begin{equation}
\mathbb{A}^\dagger:=\mathbb{M}_\mathbb{H}^{-1}\mathbb{A}^\ast\mathbb{M}_\mathbb{H}.
\end{equation}

\subsection{Particular Ladder Operators} For a given finite dimensional quantum dynamical system with space of states $\mathcal{H}:=\mathcal{H}(G)$, one can find or obtain an orthonormal system $\hat{X}:=\{\ket{k},1\leq k\leq N\}\subset \mathcal{H}$ with respect to $\braket{\cdot|\cdot}$ that will be called particular analysis basis, and that is related to a given representative graph $G:=(V_G,E_G)$ with $V_G:=\{1,\cdots,N\}$ and $E_G:=\{(\braket{ij},w_{ij})\}$ according to the rule
\begin{equation}
k \longleftrightarrow \ket{k}.
\end{equation}
\begin{remark}
It is important to note that an important idea behind the particular analysis basis, is to find a basis that provides some advantage for the analysis of the quantum dynamical system in wich space of states we define it.
\end{remark}

Once we get an orthonormal system $\hat{X}\subset \mathcal{H}$ it is possible to obtain a matrix $\mathbb{N}\in M_N(\mathbf{C})$ that satisfies $\mathbb{N}\ket{n}=n\ket{n}$ and will be defined by
\begin{equation}
\mathbb{N}:=\sum_{k=1}^N k\ket{k}\bra{k}
\end{equation}
in a similar way one can define two particular ladder operators $a^\dagger, a \in M_N(\mathbf{C})$ defined implicitly by 
\begin{eqnarray}
a^\dagger\ket{n}&:=&\sqrt{ [n+1|N]}\ket{[n+1|N]}\\
a\ket{n}&:=&\sqrt{[n|N]}\ket{[n-1|N]}
\end{eqnarray}
where $[p|q]:=1+((p-1)\mod q)$. From the implicit definition of the particular Ladder operators we can obtain the following explicit definitions
\begin{eqnarray}
a^\dagger&:=&\ket{1}\bra{N}+\sum_{k=1}^{N-1} \sqrt{k+1}\ket{k+1}\bra{k}\\
a&:=&\ket{N}\bra{1}+\sum_{k=2}^N \sqrt{k} \ket{k-1}\bra{k}.
\end{eqnarray}
It can be seen that $\mathbb{N}=a^\dagger a$. Since we will have that $\hat{X}\subset \mathcal{H}$ is an orthonormal system,  also that $\mathbb{H}_0^\dagger=\mathbb{H}_0$, and $\ket{m}=(m!)^{-1/2}(a^\dagger)^m\ket{1}$, and if we take $\mathbb{E}_1:=\ket{1}\bra{1}$, then we can express $\mathbb{H}_0\in M_N(\mathbf{C})$ in the form:
\begin{eqnarray}
\mathbb{H}_0&:=&\sum_{\braket{kl}} \braket{k|\mathbb{H}_0|l} \ket{k}\bra{l} \label{gen_hamilton} \\
&=&\sum_{\braket{kl}} \frac{\braket{k|\mathbb{H}_0|l}}{\sqrt{k!l!}}(a^\dagger)^k \mathbb{E}_1 a^l\\
&=&\sum_{\braket{k}} \frac{\braket{k|\mathbb{H}_0|k}}{k!}(a^\dagger)^k \mathbb{E}_1 a^k +\sum_{\braket{k>l}} \frac{\braket{k|\mathbb{H}_0|l}}{\sqrt{k!l!}}((a^\dagger)^k \mathbb{E}_1 a^l+(a^\dagger)^l \mathbb{E}_1 a^k)
\end{eqnarray}
that will be called particular representation of $\mathbb{H}_0\in M_N(\mathbf{C})$ with respect to $\hat{X}\subset \mathcal{H}$.

\section{Time Evolution and Matrix Schr{\"o}dinger Unitary Groups}
For a finite dimensional quantum dynamical system with state of spaces $\mathcal{H}(G)$, whose time evolution is modeled by the equation
\begin{equation}
\left \{
\begin{array}{l}
i\ket{\psi'(t)}=\mathbb{H}\ket{\psi(t)}\\
\ket{\psi(0)}=\ket{\psi_0}
\end{array}
\right .
\label{ivp}
\end{equation}
one can obtain an expression for its wave function $\ket{\psi(t)}$ using the matrix valued function $U_t:\mathbf{R}\longrightarrow \mathcal{H}:t\longmapsto e^{-it\mathbb{H}}$, in the form
\begin{equation}
\ket{\psi(t)}:=e^{-it\mathbb{H}}\ket{\psi_0}
\end{equation}
in many applications $\mathbb{H}\in M_N(\mathbf{C})$ is a structured matrix obtained using several matrix operations between matrices of lower order. In the following subsection we will consider particularly important cases of interacting and non-interacting quantum systems.

\subsection{Matrix Hamiltonians} For a sequence of finite dimesional Hilbert spa-ces of the form $\{\mathcal{H}_\alpha\}_{\alpha=1}^M$, if we can obtain a particular analysis bases sequence $\{\hat{X}_\alpha\}$, such that the non-interacting Hamiltonian of each finite quantum dynamical system whose space of states is defined by $\mathcal{H}_\alpha$ can be particularly represented by $\mathbb{H}_\alpha$ and defined according to \eqref{gen_hamilton}, obtaining a sequence $\{\mathbb{H}_\alpha\}$ that can be used to compute the non-interacting part an interacting Hamiltonian $\mathbb{H}\in M_N(\mathbf{C})$ using the following expression
\begin{equation}
\mathbb{H}_0:=\bigoplus_{\braket{\alpha}} \mathbb{H}_\alpha=\sum_{\braket{\alpha}} \{\mathbb{H}_\alpha\}^{\odot e_\alpha}
\end{equation}
where $\{e_\alpha:=(\delta_{r,\alpha})_r\}_\alpha\subset (\mathbf{Z}_0^+)^M$ is the canonical basis for the space of multi-indexes of length $M\in\mathbf{Z}^+_0$, and where in general for any multiindex $r\in(\mathbf{Z}^+_0)^M$, $\{\mathbb{H}_\alpha\}^{\odot r}$ is defined by
\begin{equation}
\{\mathbb{B}_\alpha\}^{\odot r}:=\bigotimes_{\braket{\alpha}}\mathbb{B}_\alpha^{r_\alpha}
\end{equation}
here $\mathbf{1}_\alpha$ denotes the identity matrix in $\mathcal{H}_\alpha$. When we want to compute a solution to \eqref{ivp} and if $\mathbb{H}:=\mathbb{H}_0+\mathbb{H}_1$, where $\mathbb{H}_0$ is an non-interacting hamiltonian and $\mathbb{H}_1$ is diagonal, it is useful to compute an integrating factor that is a solution to the matrix differenial equation
\begin{equation}
\left \{
\begin{array}{l}
iU_t'=\mathbb{H}_0 U_t\\
U_0=\mathbf{1}
\end{array}
\right .
\label{ivp_matrix}
\end{equation}
A solution to this equation will have the form $U_t:=e^{-it\mathbb{H}_0}$, and clearly the set $\{U_t,t\in\mathbf{R}\}$ will be an unitary group of operators, now this integrating factor can be used to solve \eqref{ivp} using the formula
\begin{equation}
\ket{\psi(t)}=U_t\ket{\psi_0}+\int_0^t ds U_{t-s}\mathbb{V}\ket{\psi(t)}
\end{equation}
whose solvability have been discussed in \cite{Vides2}, in particular the case where both the interacting and non-interacting parts of the hamiltonian $\mathbb{H}$ can be expressed in a non-interacting Hamiltonians form, i.e., $\mathbb{H}=\mathbb{H}_0$, it is important because we can express the solution to \eqref{ivp} in the form
\begin{equation}
\ket{\psi(t)}:=U_t\ket{\psi_0}=e^{-it\mathbb{H}_0}\ket{\psi_0} \label{gen_sol}
\end{equation}
where
\begin{eqnarray}
\ket{\psi(t)}&:=&\bigotimes_{\braket{\alpha}}\ket{\psi_\alpha(t)}\\
\ket{\psi_0}&:=&\bigotimes_{\braket{\alpha}}\ket{\psi_{0,\alpha}}.
\end{eqnarray}
It is also important to note that for some complex systems modeled by non-interacting type Hamiltonians, even if their wave funtion can be expressed like \eqref{gen_sol}, the computation of $e^{-it\mathbb{H}_0}$ that can be expressed in the form 
\begin{equation}
e^{-it\mathbb{H}_0}:=\bigotimes_{\braket{\alpha}} e^{-it\mathbb{H}_\alpha}
\label{exp_kron}
\end{equation}
can become a hard computational problem, this is the reason to implement some numerical techniques that will be presented in the next section.

\subsection{Approximation of Matrix Schr{\"o}dinger Unitary Groups} When we want to compute an approximation of a particular matrix Schr{\"o}dinger unitary group $\{U_t:=e^{-it\mathbb{H}}, t\in\mathbf{R}\}$, we can start approximating $U_t$ in $[0,\tau]\subset \mathbf{R}$, for  a given $1>\tau:=h_t/\norm{}{\mathbb{H}}\in\mathbf{R}$, with $\mathbf{R}\ni h_t < 1$, this restriction for $\tau$ will help to ensure that the sum
\begin{equation}
\mathbb{U}:=\sum_{k=0}^{m}\frac{(-i\tau\mathbb{H})^k}{k!}
\label{disc_sem}
\end{equation}
remains bounded with respect to $\norm{}{\cdot}$. Clearly this sum represents the first $m$ terms of the Taylor polynomial of $U_\tau$ around $t=0$, now if we take the Pad{\'e} representation of the approximant $\mathbb{U}$ we obtain the following expression:
\begin{equation}
\mathbb{U}:=R_{pp}(-i\tau\mathbb{H})=D_{pp}(-i\tau\mathbb{H})^{-1}N_{pp}(-i\tau\mathbb{H})
\label{group_factors}
\end{equation}

with
\begin{eqnarray}
N_{pq}(-i\tau\mathbb{H})&:=&\sum_{j=0}^p\frac{(p+q-j)!p!}{(p+q)!j!(p-j)!}(-i\tau\mathbb{H})\\
D_{pq}(-i\tau\mathbb{H})&:=&\sum_{j=0}^q\frac{(p+q-j)!q!}{(p+q)!j!(q-j)!}(i\tau\mathbb{H})
\end{eqnarray}
It can be seen that taking $\mathbb{S}:=N_{pp}(-i\tau\mathbb{H})$, we will have that $D_{pp}(-i\tau\mathbb{H})=\mathbb{S}^\dagger$, and if we take $\mathbb{S}^+:=(\mathbb{S}^\dagger)^{-1}$, then we can express \eqref{group_factors} in the form
\begin{equation}
\mathbb{U}=\mathbb{S}^+\mathbb{S}
\end{equation}
From the relation of \eqref{group_factors} with the Taylor expansion of $e^{-i\tau\mathbb{H}}$ and the Picard's resctriction for $\tau\in\mathbf{R}^+$ in \eqref{disc_sem}, we can obtain the following estimate
\begin{lemma} \label{lema_conv1}
$\norm*{2}{e^{-i\tau\mathbb{H}}-\mathbb{U}}\leq\left|\frac{1}{(2p+1)!}-c_{p,2p+1}\right|h_{\tau}^{2p+1}$
\end{lemma}
\begin{proof}
\begin{eqnarray}
\norm*{2}{e^{-i\tau\mathbb{H}}-\mathbb{U}}&\leq&\norm*{2}{\sum_{k=2p+1}^\infty (\frac{1}{k!}-c_{p,k})(-i\tau\mathbb{H})^k}\\
&\leq&\left|\sum_{k=2p+1}^\infty \left|\frac{1}{k!}-c_{p,k}\right| (-\tau)^k\norm*{2}{\mathbb{H}}^k\right|\\
&\leq&\left|\frac{1}{(2p+1)!}-c_{p,2p+1}\right|\tau^{2p+1}\norm*{2}{\mathbb{H}}^{2p+1}\\
&\leq&\left|\frac{1}{(2p+1)!}-c_{p,2p+1}\right|h_{\tau}^{2p+1}.
\end{eqnarray}
\end{proof}

Since $\mathbb{H}$ will be considered in general self adjoint, i.e., $\mathbb{H}^\dagger=\mathbb{H}$, we will have that $\mathbb{H}$ is normal, hence can be factored in the form $\mathbb{H}=\mathbb{V}\mathbb{D}\mathbb{V}^\ast$, with $\mathbb{D}:=\text{diag}\{d_i\}$, and taking $\Lambda:=R_{pp}(-i\tau\mathbb{D})$ we obtain
\begin{eqnarray}
\mathbb{U}&=&\mathbb{S}^+\mathbb{S}\\
&=&\mathbb{V}\Lambda^+\mathbb{V}^\ast\mathbb{V}\Lambda\mathbb{V}^\ast\\
&=&\mathbb{V}\Lambda^+\Lambda\mathbb{V}^\ast
\end{eqnarray}
wich implies the following result.
\begin{lemma}
$\mathbb{U}^\ast\mathbb{U}=\mathbf{1}$.
\end{lemma}
\begin{proof}
\begin{eqnarray}
\mathbb{U}^\ast\mathbb{U}&=&\mathbb{V}\Lambda^\ast\Lambda^{-1}\mathbb{V}^\ast\mathbb{V}\Lambda^+\Lambda\mathbb{V}^\ast\\
&=&\mathbb{V}\Lambda^\ast\Lambda^{-1}\Lambda^+\Lambda\mathbb{V}^\ast\\
&=&\mathbb{V}\Lambda^\ast\Lambda^+\Lambda^{-1}\Lambda\mathbb{V}^\ast\\
&=&\mathbb{V}\mathbb{V}^\ast\\
&=&\mathbf{1}.
\end{eqnarray}
\end{proof}
From the above relations we can see that the operator
\begin{equation}
\hat{\mathbb{U}}:=\mathbb{W}_{\mathbb{H}}^{-1}\mathbb{U}\mathbb{W}_\mathbb{H}
\label{norm_ev}
\end{equation}
satisfies the following relations
\begin{equation}
\norm{}{\hat{\mathbb{U}}\phi}=\norm*{2}{\mathbb{W}_\mathbb{H}\mathbb{W}_\mathbb{H}^{-1}\mathbb{U}\mathbb{W}_\mathbb{H}\phi}=\norm*{2}{\mathbb{U}\mathbb{W}_\mathbb{H}\phi}=\norm*{2}{\mathbb{W}_\mathbb{H}\phi}=\norm{}{\phi}
\end{equation}
wich implies that $\|\hat{\mathbb{U}}\|=1$, the adjoint of $\hat{\mathbb{U}}$ can be obtained in the following way:
\begin{eqnarray}
\hat{\mathbb{U}}^\dagger&=&\mathbb{M}_{\mathbb{H}}^{-1}\hat{\mathbb{U}}^\ast\mathbb{M}_{\mathbb{H}}\\
&=&\mathbb{M}_{\mathbb{H}}^{-1}(\mathbb{W}_{\mathbb{H}}^{-1}\mathbb{U}\mathbb{W}_\mathbb{H})^\ast\mathbb{M}_{\mathbb{H}}\\
&=&\mathbb{M}_{\mathbb{H}}^{-1}\mathbb{W}_{\mathbb{H}}^\ast\mathbb{U}^\ast\mathbb{W}_\mathbb{H}^+\mathbb{M}_{\mathbb{H}}\\
&=&\mathbb{W}_{\mathbb{H}}^{-1}\mathbb{W}_{\mathbb{H}}^{+}\mathbb{W}_{\mathbb{H}}^\ast\mathbb{U}^\ast\mathbb{W}_\mathbb{H}^+\mathbb{W}_{\mathbb{H}}^\ast\mathbb{W}_{\mathbb{H}}\\
&=&\mathbb{W}_{\mathbb{H}}^{-1}\mathbb{U}^\ast\mathbb{W}_{\mathbb{H}}
\end{eqnarray}
the matrix Schr{\"o}dinger unitary group relative to $\mathbb{H}$, will have the form $\{\hat{\mathbb{U}}_k:=\hat{\mathbb{U}}^k(\cdot),k\in\mathbf{Z}^+\}$. It can be seen that
\begin{equation}
\hat{\mathbb{U}}^\dagger\hat{\mathbb{U}}=\mathbf{1}=\hat{\mathbb{U}}\hat{\mathbb{U}}^\dagger
\end{equation}
and this implies that the discrete time reversal Schr{\"o}dinger unitary group will be given by $\{\hat{\mathbb{U}}_{-k}:=(\hat{\mathbb{U}}^\dagger)^{k}(\cdot),k\in\mathbf{Z}^+\}$ and will be coherent with the local time reversibility of Schr{\"o}dinger unitary groups. From lemma L.\ref{lema_conv1} and taking the time interval $[0,m\tau]\subset \mathbf{R}$ we can obtain the following.

\begin{lemma} \label{lema_conv2} 
$\norm{}{e^{-im\tau\mathbb{H}}-\hat{\mathbb{U}}^m}\leq \frac{{m}^{2p+1}}{(2p+1)!}h_{\tau}^{2p+1}$.
\end{lemma}
\begin{proof}
If we denote by $c_{p,k}$ the multinomial Pad{\'e} coefficients, then we will get
\begin{eqnarray}
\norm{}{e^{-im\tau\mathbb{H}}-\hat{\mathbb{U}}^m}&=&\norm{}{\sum_{k=0}^\infty \frac{(-im\tau\mathbb{H})^k}{k!}-(\sum_{k=0}^\infty c_{p,k}(-i\tau\mathbb{H})^k)^m}\\
&\leq&\norm{}{\sum_{k=2p+1}^\infty (\frac{m^k}{k!}-c_{p,k})(-i\tau\mathbb{H})^k}\\
&\leq&\left|\sum_{k=2p+1}^\infty (\frac{m^k}{k!}-c_{p,k})(-\tau)^k\norm{}{\mathbb{H}}^k\right|\\
&\leq&\frac{m^{2p+1}}{(2p+1)!}h_\tau^{2p+1}.
\end{eqnarray}
\end{proof}

Now if we denote by $\{\hat{\mathbb{U}}_\alpha\}$ the sequence of matrix approximants corresponding to the sequence of matrix Schr{\"o}dinger basics $\{e^{-i\tau\mathbb{H}_\alpha}\}$, and if we take $\{h_\alpha\}$ to be the sequence of basic time step sizes considered for the approximation of each matrix Schr{\"o}dinger basics, and if we take $e^{-im\tau\mathbb{H}_0}$ defined according to \eqref{exp_kron} and the non-interacting approximant $\hat{\mathbb{U}}\in M_N(\mathbf{C})$ defined by
\begin{equation}
\hat{\mathbb{U}}:=\bigotimes_{\braket{\alpha}}\hat{\mathbb{U}}_\alpha
\end{equation}
we can obtain the following estimate.
\begin{theorem} \label{teorema_conv} There exists $h\in (0,1)\subset \mathbf{R}^+$, such that $\norm{}{e^{-im\tau\mathbb{H}_0}-\hat{\mathbb{U}}^m}\leq (2^M-1)m^{2p+1}/(2p+1)!h^{2p+1}$.
\end{theorem}

\begin{proof}
If we define $h:=\sup_{\alpha}{h_\alpha}$ and take $\Delta_{m,\alpha}:=(e^{-im\tau\mathbb{H}_\alpha}-\hat{\mathbb{U}}_\alpha^m)$ we can use the operation
\begin{equation}
\{\mathbb{A}_r\}{\odot}\{\mathbb{B}_s\}:=\bigotimes_{p=1}^M (\mathbb{A}_\alpha(e_p\cdot r_p)+\mathbb{B}_p (e_p\cdot s_p)) \nonumber
\end{equation}
where $r,s\in(\mathbf{Z}_0^+)^M$ are multi-indexes that that satisfy $\norm*{\infty}{r}=\norm*{\infty}{s}=1$, $|s|+|r|=M$, $r\cdot s=0$ and $|r|>0$, to express $e^{-im\tau\mathbb{H}_0}-\hat{\mathbb{U}}^m$ in the form
\begin{eqnarray}
e^{-im\tau\mathbb{H}_0}-\hat{\mathbb{U}}^m&=&\bigotimes_{\braket{\alpha}}e^{-im\tau\mathbb{H}_\alpha}-\bigotimes_{\braket{\alpha}}\hat{\mathbb{U}}_\alpha\\
&=&\sum_{\braket{|r|+|s|=M}}\{\Delta_{m,r}\}{\odot}\{\mathbb{U}^m_s\}
\end{eqnarray}
Therefore,
\begin{eqnarray}
\norm{}{e^{-im\tau\mathbb{H}_0}-\hat{\mathbb{U}}^m}&=&\norm{}{ \sum_{\braket{|r|+|s|=M}}  \{\Delta_{m,r}\}{\odot}\{\hat{\mathbb{U}}^m_s\}}\\
&\leq& \sum_{\braket{|r|\leq M}} \prod_{\braket{r}}\norm{}{\Delta_{m,r}}\\
&\leq& \sum_{\braket{|r|\leq M}}\prod_{\braket{r}}\frac{m^{2p+1}}{(2p+1)!}h_r^{2p+1}\\
&\leq& \frac{(2^M-1)m^{2p+1}}{(2p+1)!}h^{2p+1}.
\end{eqnarray}
wich provides the desired result.
\end{proof}

\section{Numerical Examples}

In this section we present some basic examples to ilustrate the implementation of the ideas presented here to compute the evolution of states in a particular finite dimensional quantum dynamical system.

\begin{example}

For a double-slit experiment under simple-absorption photonic conditions, modeled by a quantum dynamical system with a space of states $\mathcal{H}\cong \mathbf{C}^5$, and described by the representative graph $G:=(V_G,E_G)$ with vertex and edge sets given by
\begin{eqnarray*}
V_G&:=&\{1,2,3,4,5\}\\ E_G&:=&\{(\braket{1;2},1),(\braket{1;3},1),(\braket{2;4},1),(\braket{2;5},1),(\braket{3;4},1),(\braket{3;5},1)\}
\end{eqnarray*}
that is isomorphic to
\begin{figure}[h]%
\caption{Representative Graph $G:=(V_G,E_G)$ of a Double-Slit Experiment.}%
\label{graph}%
\begin{center}
        \begin{tikzpicture}[node distance=12mm]
        \tikzstyle{every node}=
        [%
            minimum size=2mm,%
            inner sep=0pt,%
            outer sep=0pt,%
            ball color=black,%
          fill=black!50!black!50,%
          draw=black!50!black,%
          circle,%
        ]

        \node (1) {};
        \node (2) [right of=1] {};
        \node (3) [below of=1] {};
        \node (4) [left of=1] {};
        \node (5) [above of=1] {};


        \path [-,black,thick](1) 
                                edge  (3)
                                edge (5)  
                             (5) edge (2)
                             (2) edge (3)
                             (3) edge (4)
                             (4) edge (5);
                            
      \end{tikzpicture}
      \end{center}
\end{figure}
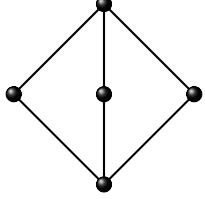
If we take $\hat{X}:=\{e_k\}\subset \mathcal{H}$ with $e_k:=(\delta_{k,j})_j$, then we will have that
\begin{eqnarray}
\mathbb{N}&:=&\left(
\begin{array}{ccccc}
1&0&0&0&0\\
0&2&0&0&0\\
0&0&3&0&0\\
0&0&0&4&0\\
0&0&0&0&5
\end{array}
\right)\\
a^\dagger&:=&\left(
\begin{array}{ccccc}
0&0&0&0&1\\
\sqrt{2}&0&0&0&0\\
0&\sqrt{3}&0&0&0\\
0&0&2&0&0\\
0&0&0&\sqrt{5}&0
\end{array}
\right)\\
a&:=&\left(
\begin{array}{ccccc}
0&\sqrt{2}&0&0&0\\
0&0&\sqrt{3}&0&0\\
0&0&0&2&0\\
0&0&0&0&\sqrt{5}\\
1&0&0&0&0
\end{array}
\right)
\end{eqnarray}
clearly $\mathbb{N}=a^\dagger a$, now, if the Hamiltonian $\mathbb{H}\in M_5(\mathbf{C})$ of the system is defined by
\begin{equation}
\mathbb{H}:=\mathbb{A}(G):=\left(
\begin{array}{ccccc}
0&1&1&0&0\\
1&0&0&1&1\\
1&0&0&1&1\\
0&1&1&0&0\\
0&1&1&0&0
\end{array}
\right)
\end{equation}
then we will have that $e^{-i\tau\mathbb{H}}:=Ve^{-i\tau\mathbb{\lambda}}V^\ast$, with
\begin{eqnarray}
V&:=&
\left(\begin{array}{ccccc} 1 & 1 & 0 & -1 & -1\\ -\frac{\sqrt{6}}{2} & \frac{\sqrt{6}}{2} & -1 & 0 & 0\\ -\frac{\sqrt{6}}{2} & \frac{\sqrt{6}}{2} & 1 & 0 & 0\\ 1 & 1 & 0 & 1 & 0\\ 1 & 1 & 0 & 0 & 1 \end{array}\right)\\
e^{-i\tau\Lambda}&:=&
\left(\begin{array}{ccccc} \mathrm{e}^{\sqrt{6}\, t\, \mathrm{i}} & 0 & 0 & 0 & 0\\ 0 & \mathrm{e}^{- \sqrt{6}\, t\, \mathrm{i}} & 0 & 0 & 0\\ 0 & 0 & 1 & 0 & 0\\ 0 & 0 & 0 & 1 & 0\\ 0 & 0 & 0 & 0 & 1 \end{array}\right).
\end{eqnarray}
also we will have that the Third order Pad{\'e} approximant that coincides with the Crank-Nicholson scheme will have the form
\begin{equation}
\hat{\mathbb{U}}:=
\left(\begin{array}{ccccc} \frac{\tau^2 + 2}{3\, \tau^2 + 2} & -\frac{2\, \tau\, \mathrm{i}}{3\, \tau^2 + 2} & -\frac{2\, \tau\, \mathrm{i}}{3\, \tau^2 + 2} & -\frac{2\, \tau^2}{3\, \tau^2 + 2} & -\frac{2\, \tau^2}{3\, \tau^2 + 2}\\ -\frac{2\, \tau\, \mathrm{i}}{3\, \tau^2 + 2} & \frac{2}{3\, \tau^2 + 2} & -\frac{3\, \tau^2}{3\, \tau^2 + 2} & -\frac{2\, \tau\, \mathrm{i}}{3\, \tau^2 + 2} & -\frac{2\, \tau\, \mathrm{i}}{3\, \tau^2 + 2}\\ -\frac{2\, \tau\, \mathrm{i}}{3\, \tau^2 + 2} & -\frac{3\, \tau^2}{3\, \tau^2 + 2} & \frac{2}{3\, \tau^2 + 2} & -\frac{2\, \tau\, \mathrm{i}}{3\, \tau^2 + 2} & -\frac{2\, \tau\, \mathrm{i}}{3\, \tau^2 + 2}\\ -\frac{2\, \tau^2}{3\, \tau^2 + 2} & -\frac{2\, \tau\, \mathrm{i}}{3\, \tau^2 + 2} & -\frac{2\, \tau\, \mathrm{i}}{3\, \tau^2 + 2} & \frac{\tau^2 + 2}{3\, \tau^2 + 2} & -\frac{2\, \tau^2}{3\, \tau^2 + 2}\\ -\frac{2\, \tau^2}{3\, \tau^2 + 2} & -\frac{2\, \tau\, \mathrm{i}}{3\, \tau^2 + 2} & -\frac{2\, \tau\, \mathrm{i}}{3\, \tau^2 + 2} & -\frac{2\, \tau^2}{3\, \tau^2 + 2} & \frac{\tau^2 + 2}{3\, \tau^2 + 2} \end{array}\right)
\end{equation}
by L.\ref{lema_conv2} we will have that $\|e^{-im\tau\mathbb{H}}-\hat{\mathbb{U}}^m\|\leq \frac{m^3}{6}h_\tau^3$. In particular the operator $\mathbb{N}\in M_5(\mathbf{C})$ can be used to compute the expected state of the system using the expression $\lfloor\braket{\mathbb{N}}_k\rceil$, with $\lfloor q\rceil:=\{p\in\mathbf{Z}:|q-p|=\min\{|q-r|,r\in V_G\}\}$ and where
\begin{equation}
\braket{\mathbb{N}}_k:=\bra{\psi_0}\hat{\mathbb{U}}_k^\dagger\mathbb{N}\hat{\mathbb{U}}_k\ket{\psi_0}=\bra{\psi_0}\hat{\mathbb{U}}_k^\dagger a^\dagger a\hat{\mathbb{U}}_k\ket{\psi_0}=\|a\hat{\mathbb{U}}_k\psi_0\|.
\end{equation}
where $\ket{\psi_0}$ is the initial state of the system.
\end{example}

\begin{example}
In this example we will consider a quantum system with three particles that evolve in a Fock $\mathcal{H}^{\otimes3}$ space based on the space of states of the above example, i.e., $\mathcal{H}^{\otimes3}:=\mathcal{H}\otimes\mathcal{H}\otimes\mathcal{H}$, also we will consider that $\hat{X}^{\otimes3}:=\{\ket{ijk}, i,j,k\in\{1,\cdots,5\}\}$ and that the hamiltonian of the system will be an interacting hamiltonian, in this particular case will be given by $\hat{\mathbb{H}}:=\mathbb{H}\oplus\mathbb{H}\oplus\mathbb{H}+\omega_0^3\mathbf{1}^{\otimes3}$, where $\omega_0\in\mathbf{C}$ is an absorption constant related to the media where the system evolves, the group basics $e^{-i\tau\hat{\mathbb{H}}}$ and its third order Pad{\'e} approximant will be given by
\begin{eqnarray}
e^{-i\tau\hat{\mathbb{H}}}&:=&e^{i\tau\omega_0^3}(e^{-i\tau\mathbb{H}}\otimes e^{-i\tau\mathbb{H}}\otimes e^{-i\tau\mathbb{H}})\\
\hat{\mathbb{U}}&:=&(1-i\tau\omega_0^3)^{-1}(1+i\tau\omega_0^3)\mathbb{U}\otimes\mathbb{U}\otimes\mathbb{U}
\end{eqnarray}
respectively, this example is not very complex yet, but even this simple example provides an idea of how useful can be T.\ref{teorema_conv} to obtain an estimate of the form $\|e^{-im\tau\hat{\mathbb{H}}}-\hat{\mathbb{U}}^m\|\leq \frac{2m^3}{3}h^3$. In this case we can also find an expression for the expected state of the system that will have the form
\begin{eqnarray}
\lfloor\lfloor \braket{\hat{\mathbb{N}}}_k \rceil\rceil&:=&\lfloor\braket{\hat{\mathbb{N}}_1}_k\rceil \lfloor\braket{\hat{\mathbb{N}}_2}_k\rceil \lfloor\braket{\hat{\mathbb{N}}_3}_k\rceil\\
&=&\lfloor\norm{}{a\mathbb{U}\psi_{0,1}}\rceil\lfloor\norm{}{a\mathbb{U}\psi_{0,2}}\rceil\lfloor\norm{}{a\mathbb{U}\psi_{0,3}}\rceil
\end{eqnarray}
where $\ket{\psi_0}=\ket{\psi_{0,1}}\otimes\ket{\psi_{0,3}}\otimes\ket{\psi_{0,3}}$ is the initial state of the system.
\end{example}

\begin{center}
 \normalsize\scshape{Acknowledgements}
\end{center}

\medskip

I want to say Thanks: To Hashem for everything, to Mirna, for her love, for her support, for making me laugh... to Stanly Steinberg for his support, advice and for so many useful comments and suggestions, to Francisco Figeac for a great conversation about applications of algebraic graph theory and to Concepci{\'o}n Ferrufino, Rosibel Pacheco and Jorge Destephen for all their support and advice.

\end{document}